\documentclass[a4paper,reqno,12pt]{amsart}

\usepackage{amsmath, amssymb, amsthm, enumerate,amsfonts,latexsym, mathrsfs}

\usepackage[top=30mm,right=27mm,bottom=30mm,left=27mm]{geometry}

\newtheorem{theorem}{Theorem}[section]
\newtheorem{corollary}[theorem]{Corollary}
\newtheorem{lemma}[theorem]{Lemma}
\newtheorem{proposition}[theorem]{Proposition}

\newtheorem{question}{Question}

\makeatletter
\newcommand{\imod}[1]{\allowbreak\mkern4mu({\operator@font mod}\,\,#1)}
\makeatother

\newcommand{\Irr}{{\mathrm {Irr}}}

\newcommand{\Aut}{{\mathrm {Aut}}}
\newcommand{\Out}{{\mathrm {Out}}}

\newcommand{\kernel}{{\mathrm {ker}}}
\newcommand{\PSL}{{\mathrm {PSL}}}
\newcommand{\GL}{{\mathrm {GL}}}

\newcommand{\PSU}{{\mathrm {PSU}}}
\newcommand{\SL}{{\mathrm {SL}}}

\newcommand{\PGL}{{\mathrm {PGL}}}

\newcommand{\SSS}{\mathrm{S}}
\newcommand{\A}{\mathrm{A}}

\renewcommand{\bf}{\textbf}

\newcommand{\geqs}{\geqslant}

\theoremstyle{definition}

\newtheorem{problem}{Problem}

\begin{document}
\title[\textbf{Question of Dixon and Rahnamai Barghi}]{\textbf{On a question of Dixon and Rahnamai Barghi}}

\author{Sesuai Y. Madanha}
\address{Sesuai Y. Madanha, Department of Mathematics and Applied Mathematics, University of Pretoria, Private Bag X20, Hatfield, Pretoria 0028, South Africa}
\address{DST-NRF Centre of Excellence in Mathematical and Statistical Sciences (CoE-MaSS)}
\email{sesuai.madanha@up.ac.za}

\thanks{}


\date{\today}


\begin{abstract}
Let $ G $ be a finite non-solvable group with a primitive irreducible character $ \chi $ that vanishes on one conjugacy class. We show that $ G $ has a homomorphic image that is either almost simple or a Frobenius group. We also classify such groups $ G $ with a composition factor isomorphic to a sporadic group, an alternating group $ \rm{A}_{n} $, $ n\geq 5 $ or $ \PSL_{2}(q) $, where $ q\geq 4 $ is a prime power, when $ \chi $ is faithful. Our results partially answer a question of Dixon and Rahnamai Barghi.
\end{abstract}

\maketitle


\section{Introduction}\label{s:intro}
Let $ \chi $ be a non-linear irreducible character of a finite group $ G $. A well-known theorem of Burnside \cite[Theorem 3.15]{Isa06} shows that $ \chi (g)=0 $ for some $ g\in G $, that is, $ \chi $ vanishes on some element $ g $ of $ G $. Since $ \chi $ is invariant on conjugacy classes, $ \chi $ vanishes on at least one conjugacy class. Malle, Navarro and Olsson \cite{MNO00} generalised Burnside's theorem by showing that $ \chi $ vanishes on some conjugacy class of elements of prime power order.

Many authors have studied finite groups with a non-linear irreducible character $ \chi $ that vanishes on only one conjugacy class. Zhmud' \cite{Zhm79} was the first to study them. Chillag \cite[Corollary 2.4]{Chi99} showed that either $ \chi _{G'} $ is irreducible or $ G $ is a Frobenius group with a complement of order $ 2 $ and an abelian kernel of odd order. Dixon and Rahnamai Barghi \cite[Theorem 9]{DB07} obtained some partial results when $ G $ is solvable and Qian \cite{Qia07} characterised finite solvable groups with this extremal property. Recently, Burness and Tong-Viet \cite{BT-V15} studied the group when $ \chi $ is imprimitive, being induced from an irreducible character of a maximal subgroup of $ G $.  

Dixon and Rahnamai Barghi \cite{DB07} posed some questions at the end of their paper. Among them was the question below:

\begin{question}\label{Q1}
If $ G $ is a finite non-solvable group with an irreducible character that vanishes on one conjugacy class, can $ G $ have more than one non-abelian composition factor?
\end{question}

In this article we shall attempt to answer this question and also to contribute to the classification of finite groups with an irreducible character that vanishes on one conjugacy class. In order to do so we investigate non-solvable groups with a primitive irreducible character that vanishes on a unique conjugacy class. In particular, we shall establish:

\begin{theorem}\label{firsttheorem}
Let $ G $ be a finite non-solvable group. Suppose $ \chi \in \Irr (G) $ is primitive and vanishes on one conjugacy class, $ \mathcal{C} $. Let $ K=\kernel \chi $, $ Z=Z(\chi ) $. Then there exists a normal subgroup $ M $ of $ G $ such that $ \mathcal{C}\subseteq M\setminus Z $ and $ M/Z $ is the unique minimal normal subgroup of the group $ G/Z $. Moreover, one of the following holds: 
\begin{itemize}
\item[(a)] $ G/Z $ is almost simple and $ M/K $ is quasisimple.
\item[(b)] $ G/Z $ is a  Frobenius group with an abelian kernel $ M/Z $ of order $ p^{2n} $, $ M/K $ is an extra-special $ p $-group and $ Z/K $ is of order $ p $.
\end{itemize}
\end{theorem}
 
For case (a) in Theorem \ref{firsttheorem} assume that $K=1$, that is, $\chi$ is faithful. Then $G/Z$ is almost simple with socle $M/Z$ where $M$ is quasisimple. Note that $\chi_{M}$ is irreducible and if $ \mathcal{C} $ is the unique conjugacy class of zeros of $\chi$ in $G$, then $\mathcal{C}$ is the union of $M$-conjugacy classes $\mathcal{C}_{1},\dots, \mathcal{C}_{r} $ with $r\leq |G:M|=|G/Z:M/Z|\leq |\Out(M/Z)|$. Observe that all zeros of $\chi_M$ have the same order which is a power of $ p $ for some prime $p$. Also note that $ Z(G)=Z(M) $.

We thus look at this general problem:
 
\begin{problem}\label{classifyp-powerorder}
For each quasisimple group $M$, classify all faithful characters $\chi$ such that
\begin{itemize}
\item[(i)] $ \chi $ vanishes on elements of the same $p$-power order;
\item[(ii)] the number of conjugacy classes that $\chi$ vanishes on is at most the size of the outer automorphism group of $M/Z(M)$;
\item[(iii)] if $ Z(M)\not= 1 $, then $ Z(M) $ is cyclic and of $ p $-power order.
\end{itemize}
\end{problem}

We remark that condition (iii) of Problem \ref{classifyp-powerorder} is necessary by Lemma \ref{uniqueminimal}. We solve Problem \ref{classifyp-powerorder} when $ M/Z(M) $ is isomorphic to either a sporadic group, an alternating group $ \rm{A}_{n} $, $ n\geq 5 $ or $ \PSL_{2}(q) $, $ q\geq 4 $ a prime power. Before we give the result, we make some remarks on notation. Since $ \PSL_{2}(4)\cong \PSL_{2}(5)\cong \rm{A}_{5} $, $ \rm{A}_{6}\cong \PSL_{2}(9) $, $ 2{\cdot} \rm{A}_{5}\cong \SL_{2}(5) $ and $ \rm{S}_{5}\cong \PGL_{2}(5) $ we shall use them interchangeably. In Theorems \ref{classificationp-powerorder}(3) and \ref{classificationoneclass}(4), we adopt the notation used in Atlas \cite{CCNPW85}.
\begin{theorem}\label{classificationp-powerorder}
Let $ M $ be a quasisimple group such that $ M/Z(M) $ is isomorphic to either a sporadic group, an alternating group $ \A_{n} $, $ n\geq 5 $ or $ \PSL_{2}(q) $, $ q\geq 4 $ a prime power. Suppose that $ M $ has a faithful irreducible character $ \chi $ such that: 
\begin{itemize}
\item[(i)] $ \chi $ vanishes on elements of the same $p$-power order;
\item[(ii)] the number of conjugacy classes that $\chi$ vanishes on is at most the size of the outer automorphism group of $M/Z(M)$;
\item[(iii)] if $ Z(M)\not= 1 $, then $ Z(M) $ is cyclic and of $ p $-power order.
\end{itemize}
Then $ M $ is one of the following:
\begin{itemize}
\item[(1)] $ M=\PSL_{2}(5) $, $ \chi (1)=3 $ or $ \chi (1)=4 $;
\item[(2)] $ M= \SL_{2}(5) $, $ \chi (1)=2 $ or $ \chi (1)=4 $;
\item[(3)] $ M=3{\cdot} \A_{6} $, $ \chi(1)=9 $;
\item[(4)] $ M= \PSL_{2}(7) $, $ \chi(1)=3 $;
\item[(5)] $ M= \PSL_{2}(8) $, $ \chi (1)=7 $;
\item[(6)] $ M=\PSL_{2}(11) $, $ \chi(1)=5 $ or $ \chi(1)=10 $;
\item[(7)] $ M= \PSL_{2}(q)$, $ \chi(1)=q $, where $ q\geq 5 $.
\end{itemize}
\end{theorem}

\begin{theorem}\label{classificationoneclass}
Let $ G $ be a finite group with a composition factor isomorphic to either a sporadic group, an alternating group $ \A_{n} $, $ n\geq 5 $ or $ \PSL_{2}(q) $, $ q\geq 4 $ a prime power. Then $ \chi \in \Irr (G) $ is faithful, primitive and vanishes on one conjugacy class if and only if $ G $ is one of the following groups:
\begin{itemize}
\item[(1)] $ G = \PSL_{2}(5) $, $ \chi (1)=3 $ or $ \chi (1)=4 $;
\item[(2)] $ G= \SL_{2}(5) $, $ \chi (1)=2 $ or $ \chi (1)=4 $;
\item[(3)] $ G\in \{\A_{6}{:}2_{2},~ \A_{6}{:}2_{3},~ 3{\cdot} \A_{6}{:}2_{3}\} $, $ \chi(1)=9 $ for all such $ \chi\in \Irr (G) $;
\item[(4)] $ G=\PSL_{2}(7) $, $ \chi (1)= 3 $;
\item[(5)] $ G=\PSL_{2}(8){:}3 $, $ \chi (1)=7 $;
\item[(6)] $ G=\PGL_{2}(q) $, $ \chi (1)=q $, where $ q\geq 5 $.
\end{itemize}
\end{theorem} 

Lastly we partially answer Question \ref{Q1}: 

\begin{corollary}\label{Q1answer}
If $ G $ is a finite group that has a faithful irreducible character which vanishes on one conjugacy class, then $ G $ has at most one non-abelian composition factor.
\end{corollary}

 The paper is organized as follows: In Section 2 we list the preliminary results we need to prove our main results. In Section 3 we reduce the main problem to almost simple and quasisimple groups, thus proving Theorem \ref{firsttheorem}. In Section 4 we prove Theorem \ref{classificationp-powerorder}. We prove Theorem \ref{classificationoneclass} in Section 5. Lastly we prove Corollary \ref{Q1answer}.
\section{Preliminaries}
\begin{lemma}\cite[Theorem 1]{Bla94}\label{quasisimplecharacterization}
Assume that $ G $ is a quasisimple group and let $ z\in Z(G) $. Then one of the following holds:
\begin{itemize}
\item[(i)] $ \textsl{order}(z)=6 $ and $ G/Z(G)\cong \A_{6}, \A_{7}, \rm{Fi}_{22}, \PSU_{6}(2) $ or $ ^{2}\rm{E}_{6}(2); $
\item[(ii)] $ \textsl{order}(z)=6 $ or $ 12 $ and $ G/Z(G)\cong \PSL_{3}(4), \PSU_{4}(3) $ or $ \rm{M}_{22}; $
\item[(iii)] $ \textsl{order}(z)=2 $ or $ 4 $, $ G/Z(G)\cong \PSL_{3}(4) $, and $ Z(G) $ is non-cyclic;
\item[(iv)] $ z $ is a commutator.
\end{itemize}
\end{lemma}
\begin{theorem}\label{p-powerorderelements}\cite[Theorem B, Theorem 3.4 and Theorem 5.1]{MNO00}\cite[Theorem 1.2]{BO04}
Let $ G $ be a finite group and let $ \chi \in \Irr (G) $ be non-linear. Then there exists $ g\in G $ of prime-power order such that $ \chi (g)=0 $. If $ G $ is simple or a symmetric group, we can choose $ g $ to be of prime order. 
\end{theorem}

\begin{lemma}\label{trivialcentralizersfornonzeros}
Let $ \chi \in \Irr (G) $ be faithful and  $ [x, y]\in Z(G) $ for some $ x, y\in G $. If $ \chi (x)\neq 0 $, then $ [x, y]=1 $.
\end{lemma}
\begin{proof}
We use the argument in \cite[Lemma 2.1]{MNO00}. Suppose $ z=[x, y]\neq 1 $. Thus $ xz=x^{y} $ and $ \chi (x)=\chi (x^{y})=\chi (xz)=\chi (x)\lambda (z) $, where $ \lambda \in \Irr(Z(G)) $ such that $ \chi _{Z(G)}=\chi(1)\lambda $. Dividing by $ \chi (x) $ we have $ 1=\lambda (z) $. On the other hand, $ z\neq 1 $ implies that $ \lambda (z)\neq 1 $ since $ \lambda $ is faithful. The result follows from this contradiction.
\end{proof}
\begin{lemma}\label{uniqueminimal}
Let $ G $ be a finite non-solvable group. Let $ \chi $ be a faithful primitive irreducible character of $ G $ that vanishes on only one conjugacy class, $ \mathcal{C} $. Let $ Z=Z(G) $. Then: 
\begin{itemize}
\item[(i)] There exists a normal subgroup $ M $ of $ G $ such that $ \mathcal{C}\subseteq M\setminus Z $ and $ M/Z $ is the unique minimal normal subgroup of the group $ G/Z $.
\end{itemize}
Let $ N $ be a normal subgroup of $ G $.
\begin{itemize}
\item[(ii)] If $ N\cap \mathcal{C}=\emptyset $, then $ N\leqslant Z $;
\item[(iii)] If $ \chi _{N} $ is reducible, then $ N\leqslant Z $. If $ \chi _{N} $ is irreducible, then $ \mathcal{C}\subseteq N $ and $ M\leqslant N $.
\end{itemize}
Moreover, if $ Z\not= 1 $, then every non-trivial $ z\in Z $ is a commutator. In particular $ z=[x, y] $ for some $ x, y\in \mathcal{C} $ and $ Z $ is cyclic of prime power order.
\end{lemma}
\begin{proof}
We first show (ii) and the first part of (iii). Note that since $ N $ is a normal subgroup of $ G $ either $ N\cap \mathcal{C}=\emptyset $ or $ \mathcal{C}\subseteq N $. For (ii), if $ N\cap \mathcal{C}=\emptyset $, then $ \chi $ does not vanish on $ N $. Thus $ \chi _{N}=e\psi $, for some $ \psi \in \Irr (N) $ and a positive integer $ e $ (by Clifford's theorem and the primitivity of $ \chi $). This means that $ \psi $ does not vanish on $ N $ forcing $ \psi (1)=1 $ and $ N'\leqslant \kernel \psi \leqslant  \kernel \chi =1 $. Hence $ N $ is an abelian normal subgroup and since $ \chi $ is faithful and primitive, $ N\leqslant Z $ by \cite[Corollary 6.13]{Isa06} and (ii) follows.

If $ \chi _{N} $ is reducible, then $ [\chi_{N}, \chi _{N}] \geq 2 $. By \cite[Theorem 21.1]{BZ99}, we have $ 2\leq [\chi _{N}, \chi _{N}]\leqslant 1 + \frac{|\mathcal{C}\setminus N|}{|N|} $ which implies that $ \mathcal{C}\setminus N $ is not empty. Since $ N $ is normal in $ G $ we deduce that $ \mathcal{C}\cap N= \emptyset $. This means that $ N\leqslant Z $ by (ii). Hence the first part of (iii) holds.

We now prove (i). Choose $ M\triangleleft G $ such that $ M/Z $ is a minimal normal subgroup of $ G/Z $. If $ \mathcal{C} \nsubseteq M $, then $ \mathcal{C}\cap M=\emptyset $ and $ M\leqslant Z $ by (ii), a contradiction. Thus $ \mathcal{C}\subseteq M\setminus Z $. Suppose $ M/Z $ is not unique and let $ M_{1}/Z $ be a minimal normal subgroup of $ G/Z $. Then $ \mathcal{C} \subseteq M_{1}$ by using a similar argument as above and so $ \mathcal{C}\subseteq M \cap M_{1}=Z $. But  this is a contradiction since $ \mathcal{C} $ cannot be contained in $ Z $. Hence $ M/Z $ is unique and (i) follows. 

For the last part of (iii), if $ \chi _{N} $ is irreducible and $ N\cap \mathcal{C}=\emptyset $, then $ N\leqslant Z $ by (ii). Thus $ N $ is abelian, $ \chi $ is linear contradicting the fact that $ \chi $ vanishes on $ \mathcal{C} $. It follows that $ \mathcal{C}\subseteq N $. We claim that $ cz\in \mathcal{C} $ for all $ z\in Z $, $ c\in \mathcal{C} $. Suppose that $ \mathfrak{X} $ is a representation affording $ \chi $. Then $ \mathfrak{X} $ is a scalar representation on $ Z $ and $ \mathfrak{X}(z) $ is a scalar of the form $ \lambda I $ by \cite[Lemma 2.27(a)]{Isa06}. Evaluating, we get $ \chi (cz)= \textsl{trace}~ \mathfrak{X} (cz) =\textsl{trace}~ \lambda \mathfrak{X}(c)= \lambda \chi (c) =0 $, that is, $ cz\in \mathcal{C} $. We have that $ Z < N $ and $ N/Z $ is a normal subgroup of $ G $. By (i), $ M/Z $ is the only minimal normal subgroup of $ G/Z $, implying that $ M/Z\leqslant N/Z $, that is, $ M\leqslant N $ and the result follows.

Let $ Z $ be non-trivial. We show that every non-trivial element $ z $ of $ Z $ is a commutator. Now $ cz\in \mathcal{C} $ for $ c\in \mathcal{C} $. This means there exists $ g\in G $ such that $ cz=g^{-1}cg $ and therefore $ z=c^{-1}g^{-1}cg $ as required. To show the last part, suppose $ z $ is non-trivial and $ z=[x, y]=x^{-1}y^{-1}xy $, where $ x,y\in G $ and $ x\notin \mathcal{C} $. By Lemma \ref{trivialcentralizersfornonzeros}, $ z=[x, y]=1 $, a contradiction. Hence the result follows.

We know that $ Z $ is cyclic by \cite[Lemma 2.27(d)]{Isa06}. Let $ c $ be of order $ p^{r} $ for some positive integer $ r $ using Theorem \ref{p-powerorderelements}. Then $ z^{p^{r}}=c^{p^{r}}z^{p^{r}}=(cz)^{p^{r}}=(g^{-1}cg)^{p^{r}}=g^{-1}c^{p^{r}}g= 1 $ and so $ Z $ is of prime power order.
\end{proof}

\section{A reduction theorem}\label{reduction}

In this section we reduce our main problem to almost simple and quasisimple cases. In the following proposition we follow the proofs of Lemma 2.3 and Theorem 1.1 of \cite{Qia07} with $ \chi $ primitive. Let $ N $ be a normal subgroup of $ G $. Recall that $ G $ is a relative $ M $-group with respect $ N $ if for every $ \chi \in \Irr (G) $ there exists $ H $ with $ N\leqslant H\leqslant G $ and $ \sigma\in \Irr (H) $ such that $ \sigma ^{G}=\chi $ and $ \sigma _{N}\in \Irr (N) $.
\begin{proposition}\label{GKsolvable}
Under the hypothesis and notation of Lemma \ref{uniqueminimal}, suppose further that $ M/Z$ is abelian. Then $ G/Z $ is a  Frobenius group with an abelian kernel $ M/Z $ of order $ p^{2n} $, $ M $ is an extra-special $ p $-group and $ Z $ is of order $ p $.
\end{proposition}
\begin{proof}
There exists a normal subgroup $ M $ of $ G $ such that $ \mathcal{C}\subseteq M\setminus Z $ by Lemma \ref{uniqueminimal}(i). If $ \chi _{M} $ is reducible, then  $ M\leqslant Z $ by Lemma \ref{uniqueminimal}(ii). This contradicts the choice of $ M $. Hence $ \chi _{M} $ is irreducible. From \cite[Lemma 2.27(c)]{Isa06}, $ \chi _{Z} $ is reducible since $ \chi $ is non-linear. Thus $ M $ is an extra-special $ p $-group with $ Z $ of order $ p $ by \cite[Propositions 1 and 4]{DB07}.

Using \cite[Theorem 6.18]{Isa06} we have that $ \chi _{Z} = f \varphi $, where $ f^{2}=|M/Z|=p^{2n} $ for some positive integers $ f $ and $ n $ and linear character $ \varphi $ of $ Z $. It follows that $ \chi (1)=f\varphi (1)=f=p^{n} $ and hence $ \chi (1) $ is a prime power. It follows from \cite[Lemma 2.2]{Qia07} that $ p\nmid | G{:}M | $ and hence $ M $ is the unique Sylow $ p $-subgroup of $ G $.

Now we show that $ G/Z=\overline{G} $ is a Frobenius group with kernel $ M/Z=\overline{M} $. Suppose $ \textbf{C}_{\overline{M}}(\overline{x})> 1 $ for some non-trivial $ p' $-element $ \overline{x} $ of $ \overline{G} $. Let $ \overline{Y}=\langle \overline{x} \rangle \overline{M} $, $ \overline{T}=[\overline{M}, \langle \overline{x} \rangle ] $ with $ Y= \langle x \rangle M $. Then  $ \overline{Y}'= [\overline{M}\langle \overline{x}\rangle, \overline{M}\langle \overline{x}\rangle ]= [\overline{M}, \overline{M}\langle \overline{x}\rangle ][\langle \overline{x}\rangle, \overline{M}\langle \overline{x}\rangle ] ~~=[\overline{M}, \overline{M}][\langle \overline{x}\rangle, \overline{M}][\langle \overline{x}\rangle, \overline{M}][\langle \overline{x}\rangle, \langle \overline{x}\rangle ]= [\overline{M}, \langle \overline{x} \rangle]=\overline{T} $ since $ [\overline{M}, \overline{M}]=1$, $ [\langle \overline{x}\rangle, \langle \overline{x}\rangle ] =1  $. 

We claim that $ \overline{T}=\overline{Y}'< \overline{M} $. Since $ \langle \overline{x} \rangle \overline{M} $ is a semidirect product of $ \langle \overline{x} \rangle $ and $ \overline{M} $, $ \langle\overline{x} \rangle $ acts via automorphisms on $ \overline{M} $. By \cite[Theorem 4.34]{Isa08} we have $ \overline{M}= \textbf{C}_{\overline{M}}(\langle \overline{x}\rangle)\times [\overline{M}, \langle \overline{x}\rangle ] $ because $ (|\overline{M}|, |\langle \overline{x}\rangle|)=1 $. Since $ \textbf{C}_{\overline{M}}(\overline{x}) $ is non-trivial, it follows that $ \overline{Y}'< \overline{M} $ as required.

Let $ \chi _{M} = \rho $, $ \psi = \chi _{Y} $, $ \rho = \chi _{M} = \psi _{M} $, and let $ \delta $ be an irreducible constituent of $ \chi _{T} $. Observe that $ M/Z $ is an abelian chief factor of $ G $, $ \rho=\chi_{M} $  is irreducible, so it is $ G $-invariant. Moreover, $ \rho_{Z}=\chi_{Z}=\chi(1) \mu $ for some $ G $-invariant $ \mu\in \Irr(Z)$. Using \cite[Theorem 6.18]{Isa06}, we can see that $\rho$ is fully ramified in $M/Z$ and by \cite[Problem 6.3]{Isa06}, $\rho$ vanishes on $M\setminus Z$. Note that $ Z\leqslant T $. Thus $ \rho $ vanishes on $ M\setminus T $. It then follows that $ \psi (1) = \rho (1) > \delta (1) $. Note that $ Y/T $ is abelian, hence every chief factor of every subgroup of $ Y/T $ has non-square order. Also $ T \leq M $ is solvable. 

By \cite[Theorem 6.22]{Isa06}, $ Y $ is a relative $ M $-group with respect to $ T $. We have that $ \psi = \lambda ^{Y} $, where $ \lambda \in \Irr(B) $, $ T < B \leqslant Y $, and $ \lambda _{T}=\delta $. We now show that $ B < Y $. Suppose that $ Y=B $. Whence $ \rho _{T} $ is irreducible. Then $ \rho _{T}= \chi _{T}=\delta $, which is a contradiction since $ \rho (1) > \delta (1) $ by the above argument. Hence $ B<Y $.

Let $ \mathcal{C}_{1} $ be the set that $ \psi $ vanishes on. It follows that $ Y\setminus B \subseteq \mathcal{C}_{1}\subseteq \mathcal{C}\subset M $. Thus $ Y\setminus B\subseteq M $, that is, $ M\cup B = Y $, a contradiction since $ M $ and $ B $ are proper subgroups of $ Y $. Thus $ \bf{C}_{\overline{M}}(\overline{x}) $ is trivial for any non-trivial $ p' $-element $ \overline{x} $ of $ \overline{G} $. Thus $ \bf{C}_{\overline{M}}({\overline{m}})\subseteq \overline{M} $ for all non-trivial $ \overline{m}\in \overline{M} $. By \cite[Theorem 6.7]{Isa08}, $ G/Z $ is a Frobenius group with kernel $ M/Z $.

Finally we show that $ M < G $. If $ \overline{M}=\overline{G} $, that is, if $ Z $ is a maximal normal subgroup of $ G $, then $ G/Z $ is simple and abelian. Hence $ G/Z $ is cyclic, that is, $ G $ is abelian, a contradiction since $ \chi $ is non-linear. Hence the result follows.
\end{proof}
Recall that $ G^{\infty} $ denotes the solvable residual of a group $ G $.
\begin{proposition}\label{nonabeliancase}
Let $ G $ be a finite non-solvable group. Let $ \chi $ be a faithful primitive irreducible character of $ G $ that vanishes on only one conjugacy class $ \mathcal{C} $. Let $ Z=Z(G) $. Suppose that $ M/Z $ is a non-abelian minimal normal subgroup of $ G/Z $. Then $ G/Z $ is almost simple and $ M $ is quasisimple. 
\end{proposition}
\begin{proof}
By Lemma \ref{uniqueminimal}, we may assume that $ Z $ is a $ p $-group and all elements in $ \mathcal{C} $ are $ p $-elements. We claim that $ M $ is perfect. If $ \chi _{M'} $ is reducible, then $ M'\leqslant Z $ by Lemma \ref{uniqueminimal}(iii). This implies that $ M $ is solvable, a contradiction. Hence $ \chi _{M'} $ is irreducible and it follows that $ M\leqslant M' $ using Lemma \ref{uniqueminimal}(iii). Now $ M $ is perfect as claimed. Since $ M/Z $ is a non-abelian chief factor of $ G $, we may write $ M/Z=T_{1}/Z\times T_{2}/Z\times \cdots \times T_{k}/Z $, where $ T_{i}/Z $ are isomorphic non-abelian simple groups.

Suppose that $ k=1 $. Note that $ M $ is quasisimple because $ M $ is perfect. Since $ M/Z $ is a unique minimal normal subgroup of $ G/Z $, it follows that $ G/Z $ is an almost simple group with socle $ M/Z $.

Now we assume that $ k\geq 2 $ and we will work for a contradiction.

Assume that $ Z=1 $. Then $ M=T_{1}\times T_{2}\times \cdots \times T_{k} $. Since $ \chi _{M} $ is irreducible, we have that $ \chi _{M}=\theta _{1}\times \theta _{2}\times \cdots \times \theta _{k} $ where $ \theta _{i}\in \Irr(T_{i}) $. Clearly $ \theta _{1} $ is nonlinear. Let $ a_{1}\in T_{1} $ be such that $ \theta _{1}(a_{1})=0 $, and let $ a_{2}\in T_{2} $ be a $ p' $-element. We have $ \chi (a_{1})=\chi (a_{1}a_{2})=0 $. This implies that $ a_{1}, a_{1}a_{2}\in \mathcal{C} $ are $ p $-elements, a contradiction.

Assume that $ Z > 1 $ and $ T_{i}^{\infty} $ are simple for all $ i $. Note that $ T_{i}/Z $ is simple, it follows that

\begin{center}
$ T_{i}= T_{i}^{\infty} Z = T_{i}^{\infty} \times Z $.
\end{center}
Then $ M=(T_{1}^{\infty}T_{2}^{\infty}\cdots T_{k}^{\infty})\times Z $ is not perfect, a contradiction.

Assume that $ Z > 1 $ and $ T_{i}^{\infty} $ is not simple for some $ i $. We may assume that $ T_{1}^{\infty} $ is not simple. Now $ T_{1}^{\infty} $ is quasisimple since $ T_{1}^{\infty} $ is perfect and 
\begin{center}
$ T_{1}^{\infty}/Z(T_{1}^{\infty})\cong T_{1}^{\infty}/(Z\cap T_{1}^{\infty})\cong T_{1}^{\infty}Z/Z=T_{1}/Z $
\end{center}
 is simple. Now $ Z(T_{1}^{\infty})=T_{1}^{\infty}\cap Z:=\langle z_{1}\rangle > 1 $, where $ z_{1} $ is a $ p $-element because $ Z $ is a cyclic $ p $-group. We claim that $ z_{1} $ is a commutator in $ T_{1}^{\infty} $. For if $ z_{1} $ is not a commutator, then using Lemma \ref{quasisimplecharacterization}, we have that $ T_{1}^{\infty} $ must be one of the groups in cases (i), (ii), (iii). Since $ z_{1} $ is of prime power order, this rules out cases (i) and (ii). But $ \langle z_{1}\rangle $ is cyclic so case (iii) does not hold, a contradiction. Thus by Lemma \ref{uniqueminimal}, $ z_{1}=[x, y] $ for some $ x,y\in T_{1}^{\infty} $. Also Lemma \ref{trivialcentralizersfornonzeros} yields that $ x, y\in \mathcal{C}\cap T_{1}^{\infty} $. Now let $ s_{2}\in T\setminus Z $ be a nontrivial $ p' $-element. Applying Lemma \ref{trivialcentralizersfornonzeros} again, we see that $ [x, s_{2}]=[y,s_{2}]=1 $. In particular, $ ys_{2} $ is not a $ p $-element and so $ ys_{2}\notin \mathcal{C} $. This also implies by Lemma \ref{trivialcentralizersfornonzeros} that $ [x, ys_{2}]=1 $. However, since $ [x, s_{2}]=[y,s_{2}]=1 $ we have $ [x,ys_{2}]=[x,y]=z_{1} $, and this leads to a contradiction that $ z_{1}=1 $. Now the proof is complete. 
\end{proof}

\begin{theorem}\label{generalreduction}
Let $ G $ be a finite non-solvable group. Suppose $ \chi \in \Irr (G) $ is primitive and vanishes on one conjugacy class $ \mathcal{C} $. Let $ K=\kernel \chi $, $ Z=Z(\chi ) $. Then there exists a normal subgroup $ M $ of $ G $ such that $ \mathcal{C}\subseteq M\setminus Z $ and $ M/Z $ is the unique minimal normal subgroup of the group $ G/Z $. Moreover, one of the following holds: 
\begin{itemize}
\item[(a)] $ G/Z $ is almost simple and $ M/K $ is quasisimple.
\item[(b)] $ G/Z $ is a  Frobenius group with an abelian kernel $ M/Z $ of order $ p^{2n} $, $ M/K $ is an extra-special $ p $-group and $ Z/K $ is of order $ p $ with $ K $ non-solvable.
\end{itemize}
\end{theorem}
\begin{proof}
Note that $ \mathcal{C} $ is a conjugacy class of $ G/K $, $ \chi $ is faithful on $ G/K $ and $ Z/K = Z(G/K) $ by  \cite[Lemma 2.27(f)]{Isa06}. Moreover, $ \chi \in \Irr(G/K) $ is primitive, faithful and vanishes on one conjugacy class $ \mathcal{C} $. If $ G/K $ is solvable, that is, $ M/Z $ is abelian, then by Proposition \ref{GKsolvable}, (b) holds. Otherwise $ G/K $ is non-solvable. Therefore the result follows from Proposition \ref{nonabeliancase} in the case where $ M/Z $ is non-abelian.
\end{proof}

\section{Quasisimple groups with a character vanishing on elements of the same order}
In this section we prove Theorem \ref{classificationp-powerorder}.
\subsection{Sporadic Groups}
Using the Atlas \cite{CCNPW85} we have the following result:
\begin{theorem}\label{sporadic}
Let $ M $ be a quasisimple group and $ p $ a prime. Suppose that $ M $ has a faithful irreducible character $ \chi $ such that: 
\begin{itemize}
\item[(i)] $ \chi $ vanishes on elements of the same $p$-power order;
\item[(ii)] the number of conjugacy classes that $\chi$ vanishes on is at most the size of the outer automorphism group of $M/Z(M)$;
\item[(iii)] if $ Z(M)\not= 1 $, then $ Z(M) $ is cyclic and of $ p $-power order.
\end{itemize}
Then $ M/Z(M) $ is not a sporadic simple group.
\end{theorem}

\subsection{Alternating groups}
It is well known that every irreducible character of the symmetric group $ \SSS_{n} $ is characterized by a partition of $ n $ (see \cite{JK81}) and so can be identified by the corresponding partition. Thus $ \chi_{\lambda} \in \Irr (\SSS_{n}) $ (or simply $ \lambda $ if there is no confusion of what it means) is the irreducible character of $ \SSS_{n} $ corresponding to the partition $ \lambda $. The irreducible characters of the alternating group $ \A_{n} $ are obtained by restricting the $ \chi _{\lambda} $'s to $ \A_{n} $. In particular, the restriction $ (\chi_{\lambda})_{\A_{n}} $ is irreducible if and only if $ \lambda $ is not self-associated. We also identify elements of $ \SSS_{n} $ with their cycle type, for example, an $ (n-2) $-cycle with $ (n-2,1^{2}) $ or an $ (n-1) $-cycle with $ (n-1, 1) $.

Firstly we consider our problem when the center $ Z(M) $ is trivial, i.e, when $ M $ is an alternating group. We require the following results:
\begin{proposition}\label{ord6vanishingelement}
Let $ M= \A_{n} $ or $ \SSS_{n} $, $ n \geq 8 $, and let $ \chi \in \Irr (M) $. Then $ \chi (g)=0 $ for some $ g\in M $ of even order. Moreover, if $ \chi $ is of $ 2 $-power degree, we can choose $ g\in M $ such that $ \textsl{order}(g)= 4 $.
\end{proposition}
\begin{proof} The first part follows from the proof of \cite[Proposition 4.3]{LMS16}. Now suppose $ \chi $ is of $ 2 $-power degree. By \cite[Theorems 2.4 and 5.1]{BBOO01}, $ \chi _{\lambda}(1)=2^{r} $, where $ \lambda =(2^{r}, 1) $ or $ \lambda =(2, 1^{2^{r}-1}) $, and $ n=2^{r}+1 $. Now using the Murnaghan-Nakayama rule \cite[2.4.7]{JK81} we shall give the appropriate choices for $ g $. 

Either $ n\equiv 1 \pmod 4 $ or $ n\equiv 3\pmod 4 $. If $ n\equiv 3\pmod 4 $, then using the proof of \cite[Proposition 2.4]{DPSS09} we have $ g=(4, 2^{(n-5)/2}, 1) $. If $ n\equiv 1\pmod 4 $, then take $ g=(4^{2}, 2^{(n-9)/2}, 1) $.
\end{proof}

\begin{theorem}\label{a5}
Let $ M=\A_{n}$, $ n\geq 5 $. Suppose that $ M $ has a faithful irreducible character $ \chi $ such that: 
\begin{itemize}
\item[(i)] $ \chi $ vanishes on elements of the same $p$-power order;
\item[(ii)] the number of conjugacy classes that $\chi$ vanishes on is at most the size of the outer automorphism group of $M$;
\item[(iii)] if $ Z(M)\not= 1 $, then $ Z(M) $ is cyclic and of $ p $-power order.
\end{itemize}
Then $ M\cong \A_{5} $ or $ \A_{6} $.
\end{theorem}
\begin{proof}
Using the Atlas \cite{CCNPW85} we infer that the only alternating groups with the desired property are $ \A_{5} $ and $ \A_{6} $ when $ n\leq 13 $. Suppose $ n > 13 $. 

We first consider when $ \chi $ vanishes on a $ 2 $-element. Suppose that $ \chi(1) $ is not a $ 2 $-power. Then by \cite[Theorem 1.2]{BO04}, $ \chi $ vanishes on some element of odd prime order, implying that $ \chi $ vanishes on at least two conjugacy classes of elements of distinct orders, a contradiction. Hence  $ \chi(1) $ is a $ 2 $-power. The result then follows by Proposition \ref{ord6vanishingelement} and Theorem \ref{p-powerorderelements}.

Now suppose that $ \chi $ vanishes on a $ 2' $-element. By Proposition \ref{ord6vanishingelement} we have that $ \chi $ vanishes on an element of even order. Hence $ \chi $ vanishes on at least two elements of distinct orders and the result follows.
\end{proof}

 We now consider our problem when $ Z(M) $ is non-trivial. We refer to \cite[Chapter 8]{HH92} for some basic results on ordinary representation theory of $ \tilde{\SSS}_{n} $ and $ \tilde{\A}_{n} $, where $ \tilde{\SSS}_{n} $  is any one of the two double covers of $ \SSS_{n} $ and $ \tilde{\A}_{n} $ is the double cover of $ \A_{n} $. The set of partitions of $ n $ into distinct parts is denoted by $ \mathcal{D} (n) $. Let $ \lambda \in \mathcal{D}(n) $. The partition $ \lambda $ is odd if the number of even parts in $ \lambda $ is odd, otherwise $ \lambda $ is even. The faithful irreducible characters of $ \tilde{\SSS}_{n} $ correspond to partitions $ \lambda \in \mathcal{D}(n) $. In particular, if $ \lambda $ is even, then $ \chi_{\lambda} \in \Irr (\tilde{\SSS}_{n})$ splits into distinct constituents, $ \chi ^{\pm}_{\lambda} $, when restricted to $ \tilde{\A}_{n} $ and if $ \lambda $ is odd, then the two irreducible characters $ \chi ^{\pm}_{\lambda}\in \Irr (\tilde{\SSS}_{n}) $ are the same when restricted to $ \tilde{\A}_{n} $.
 
 \begin{lemma}\label{psingularvanishingelement}
Let $ M=\tilde{\A}_{n} $ and suppose that $ n\geq p+4 $, where $ p $ is an odd prime. Suppose $ \chi \in \Irr (M) $ is faithful. Then $ \chi $ vanishes on some $ p $-singular element $ g $ of $ M $. 
\end{lemma}
\begin{proof}
If $ \lambda \in \mathcal{D}(n) $ is even, then the result holds by the proof of \cite[Theorem 4.5]{LMS16}. Now suppose that $ \lambda \in \mathcal{D}(n) $ is odd. Then the characters $ \chi ^{\pm}_{\lambda}\in \Irr (\tilde{\SSS}_{n}) $ are the same and irreducible when restricted to $ \tilde{\A}_{n} $. Let $ g\in \tilde{\A}_{n} $ be an element which projects to a cycle type $ \mu = (p)(2)^{2}(1)^{n-p-4} $. Then $ \lambda \neq \mu $ and so by \cite[Theorem 8.7(i)]{HH92}, $ \chi ^{\pm}_{\lambda}(g) =0 $.
\end{proof}
\begin{theorem}\label{2a5}
Let $ M $ be a quasisimple group such that $ M/Z(M)\cong \A_{n} $, $ n\geq 5 $ and $ Z(M)\neq 1 $. Suppose that $ M $ has a faithful irreducible character $ \chi $ such that: 
\begin{itemize}
\item[(i)] $ \chi $ vanishes on elements of the same $p$-power order;
\item[(ii)] the number of conjugacy classes that $\chi$ vanishes on is at most the size of the outer automorphism group of $M/Z(M)$;
\item[(iii)] if $ Z(M)\not= 1 $, then $ Z(M) $ is cyclic and of $ p $-power order.
\end{itemize}
Then $ M\cong 2{\cdot} \A_{5} $ with $ p=2 $ or $ 3{\cdot} \A_{6} $ with $ p=3 $.
\end{theorem}
\begin{proof}
Checking in the Atlas \cite{CCNPW85} we see that the result is true when $ n\leq 13 $. Let $ n\geq 14 $. Now $ \chi $ vanishes on an element whose order is not a $ p $-power since $ \chi $ vanishes on a $ 3 $-singular element and a $ 5 $-singular element by Lemma \ref{psingularvanishingelement}. The result follows by Theorem \ref{p-powerorderelements}.
\end{proof}
\subsection{Groups of Lie type}
We consider $ \SL_{2}(q), $ $ q=p^{n} $, where $ p $ is prime and $ n $ a positive integer. The character tables of $ \SL_{2}(q) $ are found in \cite[Theorems 38.1 and 38.2]{Dor71}. Since $ \SL_{2}(4)\cong \PSL_{2}(5)\cong \A_{5} $, $ \SL_{2}(5)\cong {2\cdot} \A_{5} $ and $ \A_{6}\cong \PSL_{2}(9) $, we will not consider these cases here. The size of the outer automorphism group of finite groups of Lie type can be found in the Atlas \cite[Chapter 3]{CCNPW85}. In particular, $ |\Out(\PSL_{2}(q))|= \gcd(2, q-1){\cdot} f $ where $ q=p^{f} $, $ p $ a prime and $ f $ a positive integer. The easy result below, whose proof we shall omit, will come in handy:
\begin{lemma}\label{outerlessthanconjugacyclasses}
Let $ q=p^{f} $ where $ p \geqslant 3 $ is prime and $ f $ is a positive integer such that $ q > 32 $. Then $ 2f + 1 < (q-3)/4 $.
\end{lemma}
\begin{proposition} \label{L2p-power}
Let $ M $ be a quasisimple group such that $ M/Z(M)= \PSL_{2}(q) $, $ q=p^{n} $, where $ p $ is prime, $ n $ is a positive integer, $ q\geqslant 7 $ and $ q\neq 9 $. Suppose that $ M $ has a faithful irreducible character $ \chi $ such that: 
\begin{itemize}
\item[(i)] $ \chi $ vanishes on elements of the same $p$-power order;
\item[(ii)] the number of conjugacy classes that $\chi$ vanishes on is at most the size of the outer automorphism group of $M/Z(M)$;
\item[(iii)] if $ Z(M)\not= 1 $, then $ Z(M) $ is cyclic and of $ p $-power order.
\end{itemize}
Then $ M $ is one of the following:
\begin{itemize}
\item[(1)] $ M= \PSL_{2}(7) $, $ \chi(1)=3 $;
\item[(2)] $ M= \PSL_{2}(8) $, $ \chi (1)=7 $;
\item[(3)] $ M= \PSL_{2}(11) $, $ \chi(1)=5 $ or $ \chi(1)=10 $;
\item[(4)] $ M= \PSL_{2}(q)$, $ \chi(1)=q $.
\end{itemize}
\end{proposition}
\begin{proof}
If we consider $ \PSL_{2}(q) $, $ 7\leqslant q \leqslant 32 $, then  (1)-(4) follow from the character tables in the Atlas \cite{CCNPW85}. We may consider $ q > 32 $. We first suppose that $ q $ is odd. The character table of $ \SL_{2}(q) $ is found in \cite[Theorem 38.1]{Dor71}. The faithful characters of $ M $ are the ones labelled $ \chi _{i} $ when $ i $ is odd, $ \theta _{j} $ when $ j $ is odd, $ \xi _{1} $ and $ \xi _{2} $ when $ q\equiv 3\pmod 4 $ ($ \varepsilon = -1 $), and $ \eta _{1} $ and $ \eta _{2} $ when $ q \equiv 1 \pmod 4 $ ($ \varepsilon = 1 $). This is because $ \chi _{i}(z)=(-1)^{i}(q + 1)=-(q+1) $ and $ \chi _{i}(1)=q + 1 $, $ \theta _{j}(z)=(-1)^{j}(q - 1)=-1(q-1) $ and $ \theta _{j}(1)=q - 1 $, $ \xi _{1}(z)=\xi _{2}(z)=\frac{1}{2}\varepsilon (q + 1)=-\frac{1}{2}(q + 1) $ and $ \xi _{1}(1)=\xi _{2}(1)=\frac{1}{2}(q + 1) $, $ \eta _{1}(z)=\eta _{2}(z)=-\varepsilon \frac{1}{2}(q - 1)=-\frac{1}{2}(q - 1) $ and $ \eta _{1}(1)=\eta _{2}(1)=\frac{1}{2}(q - 1) $.

Let $ \chi \in \{ \chi _{i}~|~ i $ is odd$ \} $. Then $ \chi $ vanishes on $ (q - 1)/2 $ conjugacy classes of elements represented by $ b^{m} $, $ 1 \leqslant m \leqslant (q-1)/2 $. Hence $ \chi $ vanishes on more than $ 2f $ conjugacy classes by Lemma \ref{outerlessthanconjugacyclasses}. Since the size of the outer automorphism group of $ M/Z(M)=\PSL_{2}(q) $ is $ 2f $, $ \chi $ does not satisfy hypothesis (ii) of the statement of our proposition.
 
Let $ \chi \in \{ \theta _{j}~|~ j $ is odd$ \} $. Then $ \chi $ vanishes on $ (q-3)/2 $ conjugacy classes of elements represented by $ a^{l} $, $ 1\leqslant l\leqslant (q-3)/2 $. By the argument above $ \chi $ is not a suitable character.

Now suppose $ \chi \in \{ \xi _{i} ~|~ i=1, 2 $ and $ q \equiv 3\pmod 4 \}$. Then $ \varepsilon = -1 $ and $ \chi $ vanishes on $ (q-1)/2 $ conjugacy classes and the result follows.

Lastly, let $ \chi \in \{\eta _{i} ~|~ i=1, 2 $ and $ q \equiv 1 \pmod 4 \}$. Then $ \varepsilon = 1 $ and $ \chi $ vanishes on $ (q-3)/2 $ conjugacy classes. Again the result then follows.

Now let $ M=\PSL_{2}(q) $ with $ q $ odd. The character tables of $ \PSL_{2}(q) $ are exhibited in \cite{Ada02}. Nevertheless we continue using the notation in \cite[Theorem 38.1]{Dor71}. The faithful characters of $ M $ are those labelled $ \phi $, $ \theta _{j} $ when $ j $ is even, $ \chi _{i} $ when $ i $ is even, $ \xi _{1} $ and $ \xi _{2} $ when $ q \equiv 1 \pmod 4 $ ($ \varepsilon =1 $), and $ \eta _{1} $ and $ \eta _{2} $ when $ q \equiv 3 \pmod 4 $ ($\varepsilon =-1 $). This is because $ \phi (z)=\phi (1) = p $, $ \theta _{j} (z)=\theta (j)(1) = q -1 $, $ \chi _{i}(z)=\chi _{i}(1)=q+1 $, $ \xi _{1}(z)=\xi _{2}(z)=\xi _{1}(1)=\xi _{2}(1)=\frac{1}{2}(q + 1) $ and $ \eta _{1}(z)= \eta _{2}(z)=\eta _{1}(1)=\eta _{2}(1)=\frac{1}{2}(q-1) $.

Let us consider $ \phi $, the Steinberg character of $ M $. Then $ \phi $ vanishes on two conjugacy classes represented by $ c $ and $ d $, both of order $ p $. Now (4) of the statement of the theorem follows because the size of the outer automorphism group of $ M $ is $ 2f $.

Consider $ \chi \in \{\chi _{i} ~|~ i $ is even$ \} $. Then $ \chi $ vanishes on $ (q-1)/4 $ conjugacy classes if $ q \equiv 1\pmod 4 $ and $ \chi $ vanishes on $ (q-7)/4 + 1 = (q-3)/4 $ conjugacy classes if $ q\equiv 3\pmod 4 $ by \cite{Ada02}. Since the size of the outer automorphism group of $ M $ is $ 2f $, $ \chi $ does not satisfy hypothesis (ii) of the statement of our proposition by Lemma \ref{outerlessthanconjugacyclasses}. 

Now consider $ \chi \in \{ \theta _{j} ~|~ j $ is even$ \} $. Then $ \chi $ vanishes on $ (q-1)/4 $ conjugacy classes when $ q \equiv 1\pmod 4 $ and $ \chi $ vanishes on $ (q-3)/4 $ conjugacy classes when $ q\equiv 3\pmod 4 $, again by \cite{Ada02}. In both cases $ \chi $ vanishes on more than $ 2f $ conjugacy classes and we are done.

Let $ \chi \in \{\xi _{i} ~|~ i=1, 2 $, and $ q\equiv 1\pmod 4\} $. Then $ \varepsilon = 1 $ and $ \chi $ vanishes on $ (q-1)/4 $ conjugacy classes so $ \chi $ vanishes on more than $ 2f $ conjugacy classes by Lemma \ref{outerlessthanconjugacyclasses} and the result follows.

Now take $ \chi \in \{\eta _{i} ~|~ i=1, 2 $ and $ q \equiv 3\pmod 4\} $. Such a $ \chi $ vanishes on $ (q-3)/4 $ conjugacy classes represented by $ a ^{l} $ and so fails satisfy the hypothesis by the above argument.

Finally we now consider $ \SL_{2}(q) $ where $ q $ is even. Its character table is exhibited in \cite[Theorem 38.2]{Dor71}. Note that since $ \gcd(2, q-1)=1 $, $ M=\SL_{2}(q)=\PSL_{2}(q) $. Now (2) follows from the Atlas \cite{CCNPW85}. We may assume that $ q\geqslant 32 $. We consider first the Steinberg character $ \phi $ of $ \PSL_{2}(q) $. Then $ \phi $ vanishes on one conjugacy class $ c $. Hence (4) follows. Consider $ \chi _{i} $, $ 1\leqslant i\leqslant (q-2)/2 $. Then $ \chi _{i} $ vanishes on elements of the form $ b^{m} $, $ 1\leqslant m \leqslant q/2 $. Hence $ \chi $ vanishes on $ q/2 $ conjugacy classes. Also $ \theta_{j} $ vanishes on at least $ (q-2)/2 $ conjugacy classes. Clearly the number of conjugacy classes is more than the size of the outer automorphism group of $ M $ in all these cases. Hence the result follows.

\end{proof}
\section{Non-solvable groups with a character vanishing on one class} 
We begin this section by showing the primitivity of the characters in Propositions \ref{a5toa13} and \ref{overallclassificationLietype}.  Imprimitive  characters for quasisimple groups were considered by Hiss, Husen and Magaard in \cite{HHM15}. This description could be used to determine which characters are primitive, for at least the cases where $ G=M $. However, here we use a different approach.

If an irreducible character $ \chi $ of $ G $ is imprimitive then $ \chi = \theta ^{G} $ for some $ \theta \in \Irr(H) $ where $ H $ is a proper subgroup of $ G $. By transitivity of induction of characters, we may assume that $ H $ is a maximal subgroup of $ G $. Hence if $ \chi $ vanishes on one conjugacy class, $ G $ should be one of the groups given in the result below:
\begin{theorem}\cite[Theorem 1.6]{BT-V15}\label{imprimitive}
Let $H$ be a maximal subgroup of a finite group $G$ such that for some nonlinear imprimitive irreducible character $\chi = \theta^G$ with $\theta \in {\Irr}(H)$. Let $N=H_G$. If $ \chi $ vanishes on one conjugacy class, then one of the following holds:
\begin{itemize}\addtolength{\itemsep}{0.2\baselineskip}
\item[{\rm (i)}] $G$ is a Frobenius group with an abelian odd-order kernel $H=G'$ of index $ 2 $;
\item[{\rm (ii)}] $G/N$ is a $2$-transitive Frobenius group with an elementary abelian kernel $M/N$ of order $p^n$ for some prime $p$ and integer $n\geqs 1$, and a complement $H/N$ of
order $p^n-1$. Moreover, $M'=N$ and one of the following holds:

\vspace{1mm}

\begin{itemize}\addtolength{\itemsep}{0.2\baselineskip}
\item[{\rm (a)}] $M$ is a Frobenius group with kernel $M'$ and $p^n=p>2$;
\item[{\rm (b)}] $M$ is a Frobenius group with kernel $K\triangleleft G$ such that $G/K\cong \SL_2(3)$ and $M/K\cong \rm{Q}_8$;
\item[{\rm (c)}] $M$ is a Camina $p$-group;
\end{itemize}

\item[{\rm (iii)}] $G/N\cong \PSL_{2}(8){:}3$, $H/N\cong \rm{D}_{18}{:}3$, $N$ is a nilpotent $7'$-group;

\item[{\rm (iv)}] $G/N\cong \A_{5}$, $H/N\cong \rm{D}_{10}$, $N$ is a $2$-group.
\end{itemize}
\end{theorem}
Using the result above, of the characters that were obtained in Propositions \ref{a5toa13} and \ref{overallclassificationLietype} we need to check primitivity of the characters in $ \PGL_{2}(8) $ and $ \A_{5} $ only. For $ \A_{5} $, $ |G{:}H|=6 $ and in Proposition \ref{a5toa13} we do not have characters of degree greater than or equal to $ 6 $. Also $ |(\PSL_{2}(8){:}3){:}(\rm{D}_{18}{:}3)|=14 $ but we have no characters of degree greater than or equal to $ 14 $ in Proposition \ref{overallclassificationLietype}. Hence the result follows:

\begin{theorem}\label{charactersareprimitive}
The characters in Propositions \ref{a5toa13} and \ref{overallclassificationLietype} are primitive. 
\end{theorem}

\subsection{Symmetric and alternating groups}

\begin{proposition}\label{a5toa13}
Let $ G $ be a finite group with a composition factor isomorphic to $ \A_{n} $, $ n\geq 5 $. Then $ G $ has a faithful irreducible character vanishing on one conjugacy class if and only if $ G $ is one of the following:
\begin{itemize}
\item[(1)] $ G = \A_{5}  $, $ \chi _{2}(1)=\chi _{3}(1)=3 $ or $ \chi _{4}(1)=4 $;
\item[(2)] $ G={2\cdot} \A_{5} $, $ \chi _{6}(1)=\chi _{7}(1)=2 $ or $ \chi _{8}(1)=4 $;
\item[(3)] $ G=\SSS_{5} $, $ \chi _{6}(1)=\chi _{7}(1)=5 $;
\item[(4)] $ G\in \{\A_{6}{:}2_{2},~ \A_{6}{:}2_{3},~ 3{\cdot} \A_{6}{:}2_{3}\} $, $ \chi(1)=9 $ for all such $ \chi\in \Irr (G) $.
\end{itemize}
\end{proposition}
\begin{proof}
By Theorem \ref{generalreduction} we have that there exist normal subgroups $ M $ and $ Z $ such that $ G/Z $ is almost simple and $ M $ is quasisimple with $ \chi _{M} $ irreducible and $ \chi _{M} $ vanishing on $ \mathcal{C}_{1}, \mathcal{C}_{2},\ldots, \mathcal{C}_{m} $ with $ m\geqslant 1 $ such that $ \mathcal{C}=\bigcup _{i=1}^{m}\mathcal{C}_{i} $. By the argument preceding Problem \ref{Q1} in the introduction, it is sufficient to only consider groups $ G $ such that $ M $ is isomorphic to the ones in the statements of Theorem \ref{a5} and Theorem \ref{2a5}. This means that $ M $ is isomorphic to $ \A_{5} $, $ \A_{6} $, $ 2{\cdot} \A_{5} $ or $ 3{\cdot} \A_{6} $. Using \textsl{GAP} \cite{GAP16} or Atlas \cite{CCNPW85} the result follows. Lastly by Theorem \ref{charactersareprimitive}, all the characters in the statement of the proposition are primitive.
\end{proof}

\subsection{Almost simple groups of Lie type}
We note that $ \PGL_{2}(q)=\PSL_{2}(q)\rtimes \langle \delta\rangle $ where $ \delta $ is a diagonal automorphism and $ |\langle \delta\rangle|=2 $. Also, $ \Aut(\PSL_{2}(q))=\PGL_{2}(q)\rtimes \langle \varphi \rangle $, where $ \varphi $ is a field automorphism and $ |\langle \varphi \rangle|=f $, where $ f $ is a positive integer.
\begin{proposition}\label{overallclassificationLietype}
Let $ G $ be a finite group with a composition factor isomorphic to $ \PSL_{2}(q) $, where $ q\geq 4 $ is a prime power. Then $ G $ has a faithful irreducible character vanishing on one conjugacy class if and only if $ G $ is one of the following:
\begin{itemize}
\item[(a)] $ G= \PSL_{2}(7) $, $ \chi (1)= 3 $;
\item[(b)] $ G= \PSL_{2}(8){:}3 $, $ \chi (1)=7 $;
\item[(c)] $ G= \PGL_{2}(q) $, $ \chi (1)=q $.
\end{itemize}
\end{proposition}
\begin{proof} By Theorem \ref{generalreduction}, we have that there exist normal subgroups $ M $ and $ Z $ such that $ G/Z $ is almost simple and $ M $ is quasisimple. By the argument in the proof of Proposition \ref{a5toa13}, $ M $ is isomorphic to one of the groups listed in Proposition \ref{L2p-power}. Suppose $ M\cong \PSL_{2}(7) $, $ \PSL_{2}(8) $ or $ \PSL_{2}(11) $. Using Atlas \cite{CCNPW85}, (a) and (b) follow.

We now consider the case in Proposition \ref{L2p-power}(4). When $ q $ is even, $ G=M=\PSL_{2}(q)=\PGL_{2}(q) $ and (c) follows from the proof of Theorem \ref{L2p-power}(4). Now suppose $ q $ is odd and let $ M=\PSL_{2}(q) $ and $ G=\PSL_{2}(q) $. Note that for $ \PGL_{2}(q) $ the Steinberg character $ \phi $ has values $ \pm 1 $ for all elements outside $ \PSL_{2}(q) $ by \cite[Section 2]{Ste51}. Also $ \PGL_{2}(q) $ has one conjugacy class of order $ p $. That the Steinberg character extends to $ \PGL_{2}(q) $ follows from \cite{Fei93}. Hence $ G=\PGL _{2}(q) $ satisfies the hypothesis as required. 

Now suppose that $ M=\PSL_{2}(q) < G \leqslant \Aut(\PSL_{2}(q)) $ and $ G\ncong \PGL_{2}(q) $. We want to show that every $ \chi $ of $ G $ vanishes on at least two conjugacy classes of $ G $. In light of Proposition \ref{L2p-power}(4), we need only consider the Steinberg character $ \phi $ of $ \PSL_{2}(q) $. By \cite[Lemma 2.2]{Qia07} we have $ \gcd(|G{:}\PSL_{2}(q)|, q)=1 $. Hence $ \phi $ is $ p $-defect zero in $ G $. We show that $ \PGL _{2}(q)\leqslant H $. By \cite[Lemma 3.1]{Whi13} the action of $ \delta $ makes the conjugacy classes represented by $ c $ and $ d $ of $ \PSL _{2}(q) $ into one conjugacy class. On the other hand, by \cite[Lemma 3.2]{Whi13} for $ 1\leq k < f $, $ \varphi ^{k} $ fixes these conjugacy classes, so $ G $ has two conjugacy classes of elements of order $ p $. Therefore $ G $ necessarily contains $ \delta $ and hence $ \PGL_{2}(q) $ and has only one conjugacy class of order $ p $, $ \mathcal{C} $ say. Thus $ G=G\cap \PGL_{2}(q)\rtimes \langle \varphi \rangle $. Now $ |\textbf{C}_{\PGL_{2}(q)}(c)|=\dfrac{|\PGL_{2}(q)|}{|\mathcal{C}|} $. This means that $ |\textbf{C}_{G}(c)|=\dfrac{|G|}{|\mathcal{C}|}=\dfrac{|G{:}\PGL_{2}(q)||\PGL_{2}(q)|}{|\mathcal{C}|} $. Since $ \gcd(|G{:}\PGL_{2}(q)|, q)=\gcd(|\langle \varphi \rangle| , q)= 1 $, there must exist $ x\in G\setminus \PGL_{2}(q)$ of order $ r\nmid q $, where $ r $ is a prime. Note that $ x\in \textbf{C}_{G}(c) $. Then $ cx $ has order $ pr $. Since $ \phi $ is $ p $-defect zero in $ G $, $ \phi $ vanishes on $ cx $. Since $ \phi $ vanishes on $ c $, we have that $ \phi $ vanishes on two distinct classes of $ G $ as required.

By Theorem \ref{charactersareprimitive}, all the characters in the statement of the proposition are primitive
\end{proof}

\section{A question of Dixon and Rahnamai Barghi}
We restate and prove Corollary \ref{Q1answer}.
\begin{corollary}
Let $ G $ be a finite non-solvable group. If $ \chi $ is a faithful irreducible character of $ G $ that vanishes on one conjugacy class, then $ G $ has only one non-abelian composition factor.
\end{corollary}
\begin{proof}
If $ \chi $ is primitive, then we consider the cases in Theorem \ref{generalreduction}. By the proof of Theorem \ref{generalreduction}, case (b) is when $ G $ is solvable. For case (b), the result follows since $ \Out(M/Z) $ is solvable.

 Suppose that $ \chi $ is imprimitive. By Theorem \ref{imprimitive}, the non-solvable cases are Theorem \ref{imprimitive}(ii)(c), (iii) and (iv). For (ii)(c) it is well known that if $ H/N $ is a non-solvable complement, then it has only one non-abelian composition factor. For (iii) and (iv) it is clear that $ G $ has only one non-abelian composition factor. Hence the result follows.
\end{proof}

Note that for the imprimitive case $ \chi $ need not be faithful.

\section*{Acknowledgements}
The author would like to thank Prof. Tong-Viet for the careful reading of this manuscript. His comments improved the presentation of the results. The author is also grateful to Prof. van den Berg and Dr. Le for their many helpful discussions during the preparation of this article. Additionally, the author thanks the referee for useful remarks and for suggesting a shorter proof of Proposition \ref{nonabeliancase}. The author acknowledges the support of DST-NRF Centre of Excellence in Mathematical and Statistical Sciences (CoE-MaSS). Opinions expressed and conclusions arrived at are those of the author and should not necessarily to be attributed to the CoE-MaSS.


\begin{thebibliography}{100}
\bibitem{Ada02}  Adams, J. (2002). \emph{Character Tables of $ \PGL(2) $, $ \SL(2) $, $ \PGL(2) $ and $ \PSL(2) $ over a finite field}. unpublished notes. available at \text{www.math.umd.edu/\%7Ejda/characters/characters.pdf}.
\bibitem{BBOO01} Balog, A., Bessenrodt, C., Olsson, J., Ono., J. (2001). Prime power degree representations of the
symmetric and alternating groups. \emph{J. Lond. Math. Soc.} \textbf{64}(2):344--356.
\bibitem{BZ99} Berkovich, Y. G., Zhmud, E. M. (1999). \emph{Characters of Finite Groups Part (2)}. Amer. Math. Soc.,Rhode Island.
\bibitem{BO04} Bessenrodt, C., Olsson., J. (2004). Weights of partitions and character zeros, \emph{Electron. J. Comb.} \textbf{11}(2):R5.
\bibitem{Bla94} Blau, H. (1994). A fixed point theorem for central elements in quasisimple groups. \emph{Proc. Amer. Math. Soc.} \textbf{122}:79--84.
\bibitem{BT-V15} Burness, T. C., Tong-Viet., H. P. (2015). Derangements in primitive permutation groups, with
an application to character theory. \emph{Q. J. Math.} \textbf{66}:63--96.
\bibitem{Chi99} Chillag, D. (1999). On zeros of characters of finite groups. \emph{Proc. Amer. Math. Soc.} \textbf{127}:977--983.
\bibitem{CCNPW85} Conway, J. H., Curtis, R. T., Norton, S. P., Parker, R. A., Wilson, R. A.(1985). \textit{Atlas of Finite Groups}. Clarendon Press.
\bibitem{DB07} Dixon, J. D., Rahnamai Barghi. A. (2007). Irreducible characters which are zero on only one conjugacy class. \emph{Proc. Amer. Math. Soc.} \textbf{135}:41--45.
\bibitem{DPSS09} Dolfi, S., Pacifici, E., Sanus, L., Spiga, P. (2009). On the orders of zeros of irreducible characters. \emph{J. Algebra} \textbf{321}:345--352.
\bibitem{Dor71} Dornhoff, L. (1971). \textit{Group Representation Theory. Part A: Ordinary Representation Theory}. Marcel Dekker, Inc., New York.
\bibitem{Fei93} Feit, W. (1993). Extending Steinberg characters, In \textit{Linear algebraic groups and their representations (Los Angeles, CA, 1992)}. Contemp. Math. 153 (American Mathematical Society,
1993):1–-9.
\bibitem{HHM15} Hiss, G., Husen, W., Magaard, K. (2015). \textit{Imprimitive irreducible modules for finite quasisimple groups}. \emph{Mem. Amer. Math. Soc.} \textbf{1104}:1--114.
\bibitem{HH92} Hoffman, P. N., Humphreys, J. F. (1992). \textit{Projective representations of the symmetric groups. $ Q $-functions and shifted tableaux}. Clarendon Press, Oxford.
\bibitem{Isa06} Isaacs, I. M. (2006). \textit{Character Theory of Finite Groups}. Amer. Math. Soc., Rhode Island.
\bibitem{Isa08} Isaacs, I. M. (2008). \textit{Finite Group Theory}. Amer. Math. Soc., Rhode Island.
\bibitem{JK81} James, G., Kerber, A. (1981). \textit{The Representation Theory of the Symmetric Group}. Addison-Wesley Publishing.
\bibitem{LMS16} Lassueur, C., Malle, G.,Schulte, E. (2016). Simple endotrivial modules for quasi-simple groups. \emph{J. Reine Angew.} \textbf{712}:141--174.
\bibitem{MNO00} Malle, G., Navarro, G., Olsson, J. B. (2000). Zeros of characters of finite groups. \emph{J. Group Theory} \textbf{3}:353--368.
\bibitem{Qia07} Qian., G. (2007). Finite solvable groups with an irreducible character vanishing on just one class of elements. \emph{Comm. Algebra} \textbf{35}(7):2235--2240.
\bibitem{QSY04} Qian, G., Shi, W., You, X. (2004). Conjugacy classes outside a normal subgroup. \emph{Comm. Algebra} \textbf{32}(12):4809--4820.
\bibitem{Ste51} Steinberg, R. (1951). The representations of $\GL(3; q)$, $\GL(4; q)$, $ \PGL(3; q)$ and $\PGL(4; q)$. \emph{Canad. J. Math.} \textbf{3}:225--235.
\bibitem{GAP16} The GAP Group. (2016). GAP-Groups, Algorithms and Programming, Version 4.8.4. http://www.gap-system.org
\bibitem{Whi13} White, D. (2013). Character degree extensions of $ \PSL_{2}(q) $ and $ \SL_{2}(q) $. \emph{J. Group Theory} \textbf{16}:1--33.
\bibitem{Zhm79} Zhmud', E. M. (1979). On finite groups having an irreducible complex character with one class of zeros. \emph{Soviet Math. Dokl.} \textbf{20}:795--797.
\end{thebibliography}
\end{document}